\documentclass[reqno]{amsart}

\usepackage{amsmath,amssymb,amsthm,mathrsfs}

\theoremstyle{plain}
\numberwithin{equation}{section}

\newcommand{\BN}{{\mathbb N}}

\newcommand{\cB}{{\mathcal B}}

\newcommand{\cF}{{\mathcal F}}
\newcommand{\cG}{{\mathcal G}}
\newcommand{\cI}{{\mathcal I}}

\newcommand{\cV}{{\mathcal V}}
\newcommand{\cW}{{\mathcal W}}\newcommand{\cX}{{\mathcal X}}
\newcommand{\cY}{{\mathcal Y}}\newcommand{\cZ}{{\mathcal Z}}

\newcommand{\wtilA}{\widetilde{A}}\newcommand{\wtilB}{\widetilde{B}}
\newcommand{\wtilC}{\widetilde{C}}\newcommand{\wtilD}{\widetilde{D}}

\newcommand{\wtilU}{\widetilde{U}}

\newcommand{\whatU}{\widehat{U}}

\newcommand{\kr}{\textup{Ker\,}}
\newcommand{\cokr}{\textup{Coker\,}}
\newcommand{\diag}{\textup{diag\,}}

\newcommand{\mat}[2]{\ensuremath{\left[\begin{array}{#1}#2\end{array} \right]}}

\newcommand{\sbm}[1]{\left[\begin{smallmatrix} #1\end{smallmatrix}\right]}

\newcommand{\wtil}[1]{{\widetilde{#1}}}
\newcommand{\what}[1]{{\widehat{#1}}}

\newcommand{\ands}{\quad\mbox{and}\quad}

\theoremstyle{plain}
\newtheorem{theorem}{Theorem}[section]
\newtheorem{corollary}[theorem]{Corollary}
\newtheorem{lemma}[theorem]{Lemma}
\newtheorem{proposition}[theorem]{Proposition}
\theoremstyle{definition}
\newtheorem{definition}[theorem]{Definition}
\newtheorem{example}[theorem]{Example}
\newtheorem{remark}[theorem]{Remark}

\newcommand{\Ind}{\textup{Ind}}

\begin{document}

\title{Equivalence after extension and Schur coupling do not coincide, on essentially incomparable Banach spaces}

\author[S. ter Horst]{S. ter Horst}
\address{S. ter Horst, Department of Mathematics, Unit for BMI, North-West
University, Potchefstroom, 2531 South Africa, and DST-NRF Centre of Excellence in Mathematical and Statistical Sciences (CoE-MaSS)}
\email{Sanne.TerHorst@nwu.ac.za}

\author[M. Messerschmidt]{M. Messerschmidt}
\address{M. Messerschmidt, Department of Mathematics and Applied Mathematics; University of Pretoria; Private bag X20 Hatfield; 0028 Pretoria; South Africa}
\email{mmesserschmidt@gmail.com}

\author[A.C.M. Ran]{A.C.M. Ran}
\address{A.C.M. Ran, Department of Mathematics, FEW, VU university Amsterdam, De Boelelaan 1081a, 1081 HV Amsterdam, The Netherlands and Unit for BMI, North-West~University, Potchefstroom, South Africa}
\email{a.c.m.ran@vu.nl}

\author[M. Roelands]{M. Roelands}
\address{M. Roelands,
School of Mathematics, Statistics \& Actuarial
Science,
Cornwallis Building, University of Kent,
Canterbury, Kent CT2 7NF,
UK}
\email{mark.roelands@gmail.com}


\thanks{This work is based on the research supported in part by the National Research Foundation of South Africa (Grant Numbers 90670 and 93406).}



\begin{abstract}
In 1994 H. Bart and   V.\'{E}. Tsekanovskii posed the question whether the Banach space operator relations matricial coupling (MC), equivalence after extension (EAE) and Schur coupling (SC) coincide, leaving only the implication EAE/MC $\Rightarrow$ SC open. Despite several affirmative results, in this paper we show that the answer in general is no. This follows from a complete description of EAE and SC for the case that the operators act on essentially incomparable Banach spaces, which also leads to a new characterization of the notion of essential incomparability. Concretely, the forward shift operators $U$ on $\ell^p$ and $V$ on $\ell^q$, for $1\leq p,q\leq \infty$, $p\neq q$, are EAE but not SC. As a corollary, SC is not transitive. Under mild assumptions, given $U$ and $V$ that are Atkinson or generalized invertible and EAE, we give a concrete operator $W$ that is SC to both $U$ and $V$, even if $U$ and $V$ are not SC themselves. Some further affirmative results for the case where the Banach spaces are isomorphic are also obtained.
\end{abstract}

\maketitle

\section{Introduction}
\setcounter{equation}{0}

Throughout this paper, let $U\in\cB(\cX)$ and $V\in\cB(\cY)$ be two (complex) Banach space operators. Here $\cB(\cV,\cW)$ stands for the Banach space of bounded linear operators from the Banach space $\cV$ into the Banach space $\cW$, abbreviated to $\cB(\cV)$ if $\cV=\cW$. The term operator will always mean bounded linear operator, and invertibility of an operator will imply that the inverse is bounded as well. Further definitions will be explained in the last three paragraphs of this introduction.

The operator relations {\em equivalent after extension (EAE)}, {\em matricial coupling (MC)} and {\em Schur coupling (SC)} for Banach space operators $U$ and $V$ where first used to solve certain integral equations \cite{BGK84}, and have found many applications since; for some recent applications, see \cite{CDS14,S17} (on diffraction theory), \cite{CP16,CP17} (on truncated Toeplitz operators), \cite{ET17} (on unbounded operator functions) and \cite{GKR17} (on Wiener-Hopf factorization). The main feature in these applications is that the relations EAE, MC and SC coincide, and that one can transfer from one to another in a constructive way. This raised the question, posed in \cite{BT94}, whether EAE, MC and SC may coincide at the level of Banach space operators. In fact, by that time it was known that EAE and MC coincide (see \cite{BGK84,BT92a}) and that they are implied by SC (see \cite{BT92b,BT94}), in short, SC $\Rightarrow$ EAE $\Leftrightarrow$ MC. Hence only the implication EAE/MC $\Rightarrow$ SC remained open. Some confirmative results were obtained in the early 1990s, for matrices \cite{BT94}, for Hilbert space Fredholm operators \cite{BT92a}, and for Banach space Fredholm operators with index 0 \cite{BT92a}. However, the main breakthroughs came in the last five years, most notably, for Hilbert space operators in \cite{T14}, initially for the separable case in \cite{tHR13}.  For Banach space operators confirmative answers were obtained for operators that can be approximated by invertible operators \cite{tHR13} and inessential (including compact and strictly singular) operators \cite{tHMRRW18}.

The importance of the Banach space geometries of $\cX$ and $\cY$ was first observed in \cite{tHMR15}. If $U$ on $\cX$ and $V$ on $\cY$ are EAE and compact, the for several Banach space properties, $\cX$ has this property if $\cY$ does, and vice versa; see Proposition 5.6 and Corollary 5.7 in \cite{tHMR15}.
Furthermore, it was shown that if $\cX$ and $\cY$ are essentially incomparable and $U$ or $V$ compact, then EAE of $U$ and $V$ cannot occur. In fact, much more is true. If $\cX$ and $\cY$ are essentially incomparable, then $U$ and $V$ are EAE precisely when they are Fredholm with $\dim \kr U=\dim \kr V$ and $\dim \cokr U=\dim \cokr V$. For SC $U$ and $V$ are also required to have index 0. These claims are part of our main result, Theorem \ref{T:SC-EssIncomp} below, and lead to the observation that EAE and SC do no coincide.

It was shown in \cite{BGKR05} that EAE and SC coincide precisely when SC is transitive, which makes SC into an equivalence relation. Thus our main result shows that SC is not transitive. In Section \ref{S:trans} for $\cX$ and $\cY$ primary and $U$ and $V$ generalized invertible, or $\cX$ and $\cY$ from a larger class of Banach spaces we call stable under finite dimensional quotients (see Section \ref{S:BSP}) and $U$ and $V$ Atkinson, we show that $U$ and $V$ are both SC to $W=U\otimes I_\cY$ (and to $W=I_\cX\otimes V$), even if $U$ and $V$ are not SC, which shows concretely that SC is not transitive. Some of the methods from Section \ref{S:trans} are employed in the final section to obtain two more cases where EAE and SC do coincide.

We now make precise, and discuss, some of the concepts used above, as well as a few that appear later in the paper. The operators $U\in\cB(\cX)$ and $V\in\cB(\cY)$ are called {\em equivalent after extension (EAE)} when there exist Banach spaces $\cX_0$ and $\cY_0$ such that $U\oplus I_{\cX_0}$ and $V\oplus I_{\cY_0}$ are equivalent, that is, when there exist invertible operators $E$ in $\cB(\cY\oplus\cY_0, \cX\oplus\cX_0)$ and $F$ in $\cB(\cX\oplus\cX_0,\cY\oplus\cY_0)$ such that
\begin{equation}\label{EAE}
\mat{cc}{U&0\\0&I_{\cX_0}}=E\mat{cc}{V&0\\0& I_{\cY_0}}F.
\end{equation}
In case EAE of $U$ and $V$ can be established with $\cX_0=\{0\}$ or $\cY_0=\{0\}$, we say that $U$ and $V$ are {\em equivalent after one-sided extension (EAOE)}. That $U$ and $V$ are EAE coincides with $U$ and $V$ being {\em matricially coupled (MC)}, which means that there exists an invertible operator $\whatU\in \cB(\cX\oplus\cY)$ with
\begin{equation}\label{MC}
\whatU=\mat{cc}{U & *\\ * & *}\ands \whatU^{-1}=\mat{cc}{* & *\\ * & V}.
\end{equation}
Moreover, $U$ and $V$ are called {\em Schur coupled (SC)} if there exists a $2 \times 2$ block operator $M=\sbm{A&B\\C&D}\in\cB(\cX\oplus\cY,\cX\oplus\cY)$ with $A$ and $D$ invertible and
\begin{equation}\label{SC}
U=A-BD^{-1}C \ands V=D-CA^{-1}B.
\end{equation}
Hence $U$ and $V$ are the Schur complements of the block operator matrix $M$ with respect to $D$ and $A$, respectively. As remarked above, SC $\Rightarrow$ EAE $\Leftrightarrow$ MC. On the other hand, EAOE implies SC \cite{BGKR05}, but the converse does not hold in general \cite{tHMR15} (see also Theorem \ref{T:SC-EssIncomp} below).

Recall that a Banach space operator $S\in\cB(\cX,\cY)$ is called {\em inessential} in case $I_\cX-TS$ is Fredholm for all $T\in\cB(\cY,\cX)$, or equivalently,  $I_\cY-ST$ is Fredholm for all $T\in\cB(\cY,\cX)$. We write $\cI(\cX,\cY)$ for the set of inessential operators in $\cB(\cX,\cY)$. Then the Banach spaces $\cX$ and $\cY$ are called {\em essentially incomparable} in case $\cI(\cX,\cY)=\cB(\cX,\cY)$. See Chapter 7 in \cite{A04} for further details, examples and references. We just mention here that $c_0$, $\ell^p=\ell^p(\BN)$ and $\ell^q=\ell^q(\BN)$, for $1\leq p,q\leq \infty$, $p\neq q$, are pairwise essentially incomparable by the Pitt-Rosenthal Theorem, and, by the recent Pitt-Rosenthal like theorem of \cite{HS}, certain discrete Morrey spaces introduced in \cite{GKS18} are essentially incomparable to $\ell^p$ as well. Note that for $S\in \cB(\cX,\cY)$ and $T\in\cI(\cY,\cX)$, the operators $I-TS$ and $I-ST$ are not only Fredholm, they also have index 0; see the proof of  Lemma 3.1 in \cite{tHMRRW18} for details. In case $\cX=\cY$ we abbreviate $\cI(\cX,\cX)$ to $\cI(\cX)$ and then $\cI(\cX)$ is an ideal in $\cB(\cX)$, in fact, the largest closed ideal in $\cB(\cX)$ for which the Fredholm operators correspond to the invertible operators in the Calkin algebra $\cB(\cX)/ \cI(\cX)$. For $\cX\neq \cY$ the inessential operators are still closed under inversion in the sense that $\cB(\cY,\cZ) \cdot \cI(\cX,\cY)\subset \cI(\cX,\cZ)$ and $\cI(\cX,\cY) \cdot \cB(\cZ,\cX)\subset \cI(\cZ,\cY)$, for any Banach space $\cZ$.

A Banach space operator $T\in\cB(\cV,\cW)$ is called {\em generalized invertible} in case $T$ has a closed complementable range and a complementable kernel. Equivalently, $T$ admits a decomposition of the form
\[
T=\mat{cc}{T'&0\\ 0&0}:\mat{c}{\cV_1\\ \cV_2}\to\mat{c}{\cW_1\\ \cW_2}\mbox{ with  $T'$ invertible}.
\]
Thus $\cV_2=\kr T$, and, with some abuse of terminology, we will usually refer to $\cW_2$ as `the' cokernel of $T$ ($\cokr T$). Note that $T$ now indeed admits a generalized inverse, namely $\sbm{T'^{-1}&0\\0&0}$ mapping $\cW=\cW_1\oplus\cW_2$ into $\cV=\cV_1\oplus\cV_2$. Now $T$ is {\em Fredholm} in case $T$ is generalized invertible and $\kr T$ and $\cokr T$ are finite dimensional and {\em Atkinson} in case it is only required that one of  $\kr T$ and $\cokr T$ is finite dimensional.

\section{SC and EAE do not coincide, on essentially incomparable Banach spaces}\label{S:EssIncomp}

In this section we prove the main result of the present paper, which is the following theorem, and give an alternative characterization of essential incomparability in Proposition \ref{P:EIchar}.

\begin{theorem}\label{T:SC-EssIncomp}
Let $U\in\cB(\cX)$ and $Y\in \cB(\cY)$ with $\cX$ and $\cY$ essentially incomparable Banach spaces. Then
\begin{itemize}
  \item[(1)] $U$ and $V$ are never EAOE;
  \item[(2)] $U$ and $V$ are SC if and only if $U$ and $V$ are Fredholm with index 0 and $\dim \kr U=\dim \kr V$ (and thus $\dim \cokr U=\dim \cokr V$);
  \item[(3)] $U$ and $V$ are EAE if and only if $U$ and $V$ are Fredholm with $\dim \kr U=\dim \kr V$ and $\dim \cokr U=\dim \cokr V$.
\end{itemize}
In particular, SC and EAE do not coincide.
\end{theorem}

\begin{proof}[\bf Proof]
Item (1) was proven in \cite{tHMR15} for $\cX=\ell^p$ and $\cY=\ell^q$, $1\leq p\neq q<\infty$, but the proof is easily adapted to the case of essentially incomparable Banach spaces, as indicated in Remark 1.3 of \cite{tHMRRW18}. The ``if'' claims in (2) and (3) follow from Theorems 3 and 4 in \cite{BT92b}, respectively, and are true without the essential incomparability of $\cX$ and $\cY$. Thus it remains to prove the ``only if'' claims in (2) and (3), and the final claim.

Starting with (2), assume $U$ and $V$ are SC, say $U$ and $V$ are as in \eqref{SC} with $\sbm{A&B\\C&D}\in\cB(\cX\oplus\cY)$ with $A$ and $D$ invertible. Then $U=A(I-A^{-1}B D^{-1}C)$ and the operators $A^{-1}B:\cY\to\cX$ and $D^{-1}C:\cX\to\cY$ are inessential, since $\cX$ and $\cY$ are essentially incomparable. Therefore, $I-A^{-1}B D^{-1}C$ is Fredholm with index 0, and hence $U$ is Fredholm with index 0. Similarly one obtains that $V$ must be Fredholm with index 0.

For the remaining implication in (3), assume $U$ and $V$ are EAE, hence MC. Thus, there exists an invertible operator $\what{U}\in\cB(\cX\oplus\cY)$ with
\[
\what{U}=\mat{cc}{U& U_{12}\\ U_{21} & U_{22}}:\mat{c}{\cX\\ \cY} \to \mat{c}{\cX\\ \cY},\quad
\what{U}^{-1}=\mat{cc}{V_{11}& V_{12}\\ V_{21} & V}:\mat{c}{\cX\\ \cY} \to \mat{c}{\cX\\ \cY}.
\]
In particular, $\what{U}$ and $\what{U}^{-1}$ are Fredholm with index 0. Since $\cX$ and $\cY$ are essentially incomparable, $U_{12}$ and $U_{21}$ are inessential. Then $\sbm{0& U_{12}\\ U_{21} & 0}$ is also inessential, and hence
\[
\mat{cc}{U&0\\ 0&U_{22}}
=\whatU - \mat{cc}{0& U_{12}\\ U_{21} & 0}
\]
is Fredholm with index 0 as well. This in turn implies $U$ and $U_{22}$ are Fredholm and $\Ind (U)+\Ind(U_{22})=0$. Similarly, it follows that $V$ is Fredholm.

A slight addition to the previous observations gives an alternative proof of the `only if' part of item (2). Indeed, as shown in \cite{BT92b}, see also \cite{tHR13}, SC of $U$ and $V$ coincides with strong MC of $U$ and $V$, which means that one can find $\whatU$ as above with $U_{22}$ and $V_{11}$ invertible. However, if $\Ind(U)\neq 0$, then $\Ind(U_{11})=-\Ind(U)\neq0$, and thus $V_{11}$ cannot be invertible. We conclude that in case $\Ind(U)\neq 0$, strong MC, and hence SC, cannot occur.

To see that EAE and SC do not coincide it suffices to find an example of essentially incomparable Banach spaces on which Fredholm operators of the same non-zero index exist. This is done in the following example.
\end{proof}

\begin{example}\label{E:shifts}
Take for both $U$ and $V$ the forward shift operator, but $U$ acting on $\ell^p$ and $V$ acting on $\ell^q$ for $1\leq p,q\leq \infty$, $p\neq q$. Then $U$ and $V$ are injective and Fredholm with index 1, hence EAE, but not SC, by Theorem \ref{T:SC-EssIncomp}. We define a concrete $\whatU$ as in \eqref{MC} that establishes the MC between $U$ and $V$. For this purpose, take $U_{12}=\diag(1,0,0,\ldots):\ell^q\to \ell^p$, $U_{21}=0$ and let $U_{22}$ be the backward shift on $\ell^q$. Then
\[
\wtilU=\mat{cc}{U & U_{21}\\ U_{12} &U_{22}}\quad\mbox{is invertible with inverse}\quad \wtilU^{-1}=\mat{cc}{V_{11}&V_{12}\\ V_{21} & V},
\]
where $V_{11}$ is the backward shift on $\ell^p$, $V_{12}=0$ and $V_{21}=\diag(1,0,0,\ldots):\ell^p\to \ell^q$. Note that for our choice of $\whatU$, we have $\Ind (U_{22})=-1=-\Ind(U)$. This happens for each choice of $\whatU$ that established the MC of $U$ and $V$, because $\ell^p$ and $\ell^q$ are essentially incomparable, and as a result $U_{22}$ cannot be invertible, and $U$ and $V$ cannot be strongly MC, and hence not SC, as observed in the above proof.
\end{example}

\begin{remark}
More generally, by Theorem 3.3 in \cite{tHR13}, any two strongly regular operators $U$ and $V$ such that the kernels and cokernels are pairwise isomorphic are SC. However, if such operators $U$ and $V$ are not Fredholm and acting on essentially incomparable Banach spaces, then this situation cannot occur, since it would lead to two infinite dimensional closed and complementable subspaces of essentially incomparable Banach spaces that are isomorphic, hence to a closed range operator of infinite rank, a contradiction.
\end{remark}

That on essentially incomparable Banach spaces only Fredholm (with index 0) operators can be EAE (SC) in fact gives a new characterization of essential incomparability.

\begin{proposition}\label{P:EIchar}
For Banach spaces $\cX$ and $\cY$ the following are equivalent:
\begin{itemize}
  \item[(1)] $\cX$ and $\cY$ are essentially incomparable;
  \item[(2)] all SC operators $U$ on $\cX$ and $V$ on $\cY$ are Fredholm (with index 0);
  \item[(3)] all EAE operators $U$ on $\cX$ and $V$ on $\cY$ are Fredholm.
\end{itemize}
The parenthesised phrase can be removed without loosing the validity of the statement.
\end{proposition}

\begin{proof}[\bf Proof]
The implications (1) $\Rightarrow$ (2, including parenthesised phrase) and (1) $\Rightarrow$ (3) follow from Theorem \ref{T:SC-EssIncomp}. We prove (3) $\Rightarrow$ (2, without parenthesised phrase) and (2, without parenthesised phrase) $\Rightarrow$ (1), which completes the proof.

Assume (3) and let $U$ and $V$ be SC. Then $U$ and $V$ are EAE, hence Fredholm by (3).

Now assume (2). Also assume $\cX$ and $\cY$ are not essentially incomparable. Then there exist operators $B\in\cB(\cY,\cX)$ and $C\in\cB(\cX,\cY)$ such that $I_\cX-CB$ is not Fredholm. Now take $M=\sbm{A&B\\ C&D}$ with $A=I_\cX$ and $D=I_\cY$. Then the Schur complement $U:=I_\cX-CB$ of $M$ is not Fredholm, and so is the other Schur complement $V:=I_\cY - BC$. Hence $U$ and $V$ are SC, but not Fredholm, in contradiction with (2). Therefore $\cX$ and $\cY$ are essentially incomparable.
\end{proof}

\section{Stability under finite dimensional quotients}\label{S:BSP}

Recall that a Banach space $\cZ$ is called {\em prime} if $\cZ$ is isomorphic to all its infinite dimensional complemented subspaces, and {\em primary} if $\cZ\simeq \cZ_1\oplus \cZ_2$ implies $\cZ\simeq \cZ_1$ or $\cZ\simeq\cZ_1$. All prime Banach spaces are primary, $\ell^p$, $1\leq p\leq\infty$, and $c_0$ are prime, and hence primary, $L^p$, $1\leq p<\infty$, and $C[0,1]$ are primary; cf., \cite{AK16}. In this section we study the following class of Banach spaces, which includes all primary Banach spaces.

\begin{definition}
A Banach space $\cZ$ is said to be {\em stable under finite dimensional quotients} if $\cZ$ is isomorphic to any subspace with a finite dimensional complement.
\end{definition}

Not all Banach spaces are stable under finite dimensional quotients, an example of one that is not is given in \cite{G94}. On the other hand, there are Banach spaces that are stable under finite dimensional quotients, but not primary. Concretely, $\ell^p$ and $\ell^q$ are prime, but $\ell^p \oplus \ell^q$ is not primary in case $p\neq q$. That $\ell^p \oplus \ell^q$ is stable under finite dimensional quotients follows from the next proposition.

\begin{proposition}\label{P:BSP-sum}
Assume $\cZ_1$ and $\cZ_2$ are stable under finite dimensional quotients. Then so is $\cZ_1\oplus\cZ_2$.
\end{proposition}

Before proving the proposition, we derive a few elementary lemmas.

\begin{lemma}\label{L:BSP-findimadd}
Let $\cZ$ be stable under finite dimensional quotients and $\cF$ finite dimensional. Then $\cZ\simeq \cZ\oplus \cF$.
\end{lemma}

\begin{proof}[\bf Proof]
Clearly $\cZ$ must be infinite dimensional. Let $\cF_0$ be a finite dimensional subspace of $\cZ$, with complement $\cZ_0$, that is isomorphic to $\cF$. Then $\cZ\simeq \cZ_0$, because $\cF_0$ is is finite dimensional, and hence
\[
\cZ\simeq \cZ_0 \oplus \cF_0 \simeq \cZ\oplus \cF. \qedhere
\]
\end{proof}

The following lemma may be well known, but we did not find it in the literature. Note that the conclusion does not hold without the finite dimensionality of $\cZ$, just take $\cZ=\ell^2$ and $\cV$ and $\cW$ finite dimensional of different dimensions.

\begin{lemma}\label{L:sumiso}
Let $\cV,\cW,\cZ$ be Banach spaces with $\cZ\oplus \cV \simeq \cZ\oplus \cW$ and $\dim \cZ<\infty$. Then $\cV\simeq \cW$.
\end{lemma}

\begin{proof}[\bf Proof]
Decompose an isomorphism $T$ from $\cZ\oplus\cV$ to $\cZ\oplus\cW$ in block operator form
\[
T=\mat{cc}{T_{11} & T_{12}\\ T_{21} & T_{22}}:\mat{c}{\cZ\\\cV}\to \mat{c}{\cZ\\ \cW}.
\]
Then $T$ is Fredholm of index 0. Thus, so is the finite rank perturbation
\[
\mat{cc}{0 &0\\ 0 & T_{22}}=T-\mat{cc}{T_{11} & T_{12}\\ T_{21} & 0}.
\]
Hence $T_{22}\in\cB(\cV,\cW)$ is a Fredholm operator of index 0. Adding a finite rank operator that maps the kernel of $T_{22}$ onto its cokernel we obtain an isomorphism between $\cV$ and $\cW$.
\end{proof}

\begin{proof}[\bf Proof of Proposition \ref{P:BSP-sum}]
Assume $\cZ_1\oplus \cZ_2=\cV\oplus\cF$ with $\cF$ finite dimensional. For $j=1,2$ let $\cF_j=P_{\cZ_j}\cF$. Then $\dim \cF_j<\infty$, hence $\cF_j$ has a complement $\cV_j$ in $\cZ_j$. Since $\cZ_j$ is stable under finite dimensional quotients, $\cZ_j\simeq \cV_j$. We have $\cF\subset \cF_1\oplus \cF_2$, with all three spaces finite dimensional. Hence $\cF_1\oplus \cF_2=\cF\oplus \cG$ for some $\cG$. Therefore, we have
\begin{align*}
\cF\oplus\cV
&=\cZ_1\oplus\cZ_2
=\cF_1\oplus\cV_1 \oplus \cF_2\oplus\cV_2
\simeq\cF_1\oplus \cF_2\oplus\cV_1 \oplus\cV_2
=\cF\oplus \cG\oplus\cV_1 \oplus\cV_2 \\
&=\cF\oplus \cG\oplus\cZ_1 \oplus\cZ_2
\simeq \cF\oplus (\cG\oplus\cZ_1) \oplus\cZ_2
\simeq\cF\oplus\cZ_1 \oplus\cZ_2,
\end{align*}
where we used Lemma \ref{L:BSP-findimadd} in the last step. Since $\cF$ is finite dimensional, by Lemma \ref{L:sumiso}, we have $\cZ_1 \oplus\cZ_2 \simeq \cV$. Hence $\cZ_1 \oplus\cZ_2$ is stable under finite dimensional quotients.
\end{proof}

\section{SC is not transitive}\label{S:trans}

By \cite[Page 215]{BGKR05} SC is transitive, and with that an equivalence relation, if and only if it coincides with EAE, which is an equivalence relation. Hence Theorem \ref{T:SC-EssIncomp} yields the following corollary.

\begin{corollary}\label{C:notEquivRel}
SC is not transitive, hence not an equivalence relation.
\end{corollary}

However, the proof in \cite{BGKR05} provides no insight into the lack of transitivity of SC. The following proposition gives two cases where EAE operators $U$ and $V$ are both SC to a third operator $W$, even if $U$ and $V$ are not SC themselves. In particular, for the operators $U$ and $V$ of Example \ref{E:shifts}, which are not SC, we obtain an operator $W$ that is SC to both $U$ and $V$, giving a concrete example that shows SC is not transitive.

\begin{proposition}\label{P:NotTrans}
Let $U$ on $\cX$ and $V$ on $\cY$ be EAE. Then there exists an operator $W$ on $\cX\oplus\cY$ such that both $U$ and $V$ are SC with $W$, concretely, for $W$ one can take either one of $U\oplus I_\cY$ or $I_\cX \oplus V$, in the following two cases:
\begin{itemize}
  \item[(1)] $U$ and $V$ Atkinson and $\cX$ and $\cY$ stable under finite dimensional quotients;
  \item[(2)] $U$ and $V$ generalized invertible and $\cX$ and $\cY$ primary.
\end{itemize}
\end{proposition}

\begin{proof}[\bf Proof]
Take $W=U\oplus I_\cY$, the case where $W=I_\cX\oplus V$ is proved analogously. Clearly $U$ and $W$ are EAOE, hence SC. In both case (1) and case (2) it remains to show $V$ and $W$ are SC. In both cases $U$ and $V$ are generalized invertible, and admit decompositions of the form
\[
U=\mat{cc}{U'& 0 \\ 0&0}:\mat{c}{\cX_1\\ \cX_2} \to \mat{c}{\cX_1'\\ \cX_2'},\quad
V=\mat{cc}{V'& 0 \\ 0&0}:\mat{c}{\cY_1\\ \cY_2} \to \mat{c}{\cY_1'\\ \cY_2'},
\]
with $U'$ and $V'$ invertible. Then $U$ being Atkinson means that $\cX_2$ or $\cX_2'$ is finite dimensional, and similar for $V$. Since $U$ and $V$ are EAE we have isomorphisms
\begin{equation}\label{T2T2'}
 T_2:\cX_2\to \cY_2 \ands T_2':\cX_2'\to \cY_2'\quad\mbox{(see \cite[Proposition 3.2]{tHR13})}.
\end{equation}
Note that
\[
U':\cX_1\to\cX_1' \ands V':\cY_1\to\cY_1'
\]
are also isomorphisms. In both cases (1) and (2) one of the following two situations always occurs:
\begin{equation}\label{cases}
\mbox{(a) } \cY_2\simeq\cX_2\simeq \cX_2'\simeq \cY_2' \quad \mbox{or} \quad
\mbox{(b) } \cX_1'\simeq\cX_1\simeq \cX \mbox{ and } \cY_1'\simeq\cY_1\simeq \cY.
\end{equation}
Indeed, if $U$ is Atkinson and $\cX$ stable under finite dimensional quotients, then the kernel $\cX_2$ or cokernel $\cX_2'$ of $U$ is finite dimensional, leading to $\cX\simeq\cX_1$ or $\cX\simeq\cX_1'$, and we already have $\cX_1\simeq \cX_1'$ via $U'$. Likewise we obtain $\cY_1'\simeq\cY_1\simeq \cY$ in case $V$ is Atkinson and $\cY$ stable under finite dimensional quotients.

Now assume we are in case (2). Since $\cX$ is primary, $\cX$ is isomorphic to $\cX_1$ or $\cX_2$ and $\cX$ is isomorphic to $\cX_1'$ or $\cX_2'$. If either $\cX\simeq \cX_1$ or $\cX\simeq\cX_1'$, then we have $\cX_1'\simeq\cX_1\simeq \cX$, as explained above. Thus we have the first set of isomorphisms in (b) or $\cX_2\simeq\cX\simeq \cX_2'$. However, we already have $\cX_2\simeq \cY_2$ and $\cX_2'\simeq \cY_2'$, hence we have either (a) or the first set of isomorphisms in (b). Reasoning in a similar fashion, using that $\cY$ is primary, we note that we have either (a) or the second set of isomorphisms in (b). Thus we are always is one of the situations (a) and (b).

We now show that in both cases (a) and (b) $V$ and $W=U\oplus I_\cY$ are SC. In case $(a)$ $U$ and $V$ are not only EAE, they are also SC, by \cite[Theorem 3.3]{tHR13}. Hence \eqref{SC} holds for some block operator $M=\sbm{A&B\\C&D}\in\cB(\cX\oplus\cY)$ with $A$ and $D$ invertible. Now simply extend $M$ to
\[
\wtil{M}=\mat{cc}{\wtilA&\wtilB\\\wtilC&\wtilD}=\mat{cc|c}{A&0&B\\ 0& I_\cY & 0\\ \hline C &0 &D}.
\]
Clearly $\wtilA$ and $\wtilD$ are invertible, and we have
\begin{align*}
V&=D-CA^{-1}B=\wtilD- \wtilC \wtilA^{-1}\wtilB,\\
W&=\mat{cc}{A-BD^{-1}C&0\\0&I_\cY}=\wtilA- \wtilB \wtilD^{-1}\wtilC.
\end{align*}
Hence $V$ and $W$ are SC.

Finally, assume we are in case (b). Then there exist isomorphisms
\begin{equation}\label{SR}
S:\cX\to\cX_1 \ands R:\cY\to\cY_1.
\end{equation}
Consequently, $U'$ and $I_\cX$ are equivalent and $V'$ and $I_\cY$ are equivalent, by the following identities
\[
I_\cX=(S^{-1}U'^{-1})U' S \ands V'= (V' R) I_\cY R^{-1}.
\]
Therefore, we have
\begin{align*}
\mat{ccc}{V' &0&0\\ 0&0&0\\ 0&0& I_\cX}
&=\mat{ccc}{0&0& V' R\\ 0&T_2'&0\\ S^{-1}U'^{-1}&0&0}
\mat{ccc}{U' &0&0\\ 0&0&0\\ 0&0& I_\cY}
\mat{ccc}{0&0& S\\ 0&T_2^{-1}&0\\ R^{-1}&0&0}
\end{align*}
with the left and right factors on the right hand side invertible. This shows that $V$ and $W$ are EAOE, which implies $V$ and $W$ are SC.
\end{proof}

The next corollary follows immediately from Theorem \ref{T:SC-EssIncomp} and Proposition \ref{P:NotTrans}.

\begin{corollary}
Let the Banach spaces $\cX$ and $\cY$ be essentially incomparable and stable under finite dimensional quotients. Then all operators $U\in\cB(\cX)$ and $V\in\cB(\cY)$ that are EAE are both SC with $U\oplus I_\cY$, as well as with $I_\cX\oplus V$.
\end{corollary}

\section{Some cases where SC and EAE do coincide}\label{S:PositiveRes}

Using similar arguments as in the previous section we can prove two new cases where SC and EAE do coincide.

\begin{proposition}\label{P:XisoY}
Let $U\in\cB(\cX)$ and $V\in\cB(\cY)$ with $\cX\simeq\cY$. Assume either (1) or (2) in Proposition \ref{P:NotTrans} holds, in particular, $U$ and $V$ are generalized invertible. Then the following are equivalent:
\begin{itemize}
  \item[(1)] $U$ and $V$ are SC;
  \item[(2)] $U$ and $V$ are EAE;
  \item[(3)] $\kr U\simeq\kr V$ and $\cokr U\simeq\cokr V$.
\end{itemize}
\end{proposition}

\begin{proof}[\bf Proof]
Since $U$ and $V$ are generalized invertible, the equivalence of (2) and (3) follows from Proposition 3.2 in \cite{tHR13}. The implication (1) $\Rightarrow$ (2) holds without additional assumptions on $U$ and $V$. Thus it remains to prove (2) $\Rightarrow$ (1). As observed in the proof of Proposition \ref{P:NotTrans}, we are either in case (a) or in case (b) of \eqref{cases}, and in case (a) $U$ and $V$ are SC, even without $\cX\simeq \cY$. Hence, we may assume that (b) in \eqref{cases} holds. Besides the isomorphisms $S$ and $R$ in \eqref{SR} and $T_2$ and $T_2'$ in \eqref{T2T2'}, since $U$ and $V$ are EAE, we now also have an isomorphism $Q:\cX\to\cY$. Then
\[
\mat{cc}{U'&0\\0&0}=\mat{cc}{U'S Q^{-1}R^{-1} &0\\0& T_2'^{-1}}\mat{cc}{V'&0\\0&0}\mat{cc}{ V'^{-1}RQS^{-1} &0\\0& T_2}=EVF,
\]
Thus $E=EVF$ holds with $E=\sbm{U'S Q^{-1}R^{-1} &0\\0& T_2'^{-1}}$ and $F=\sbm{V'^{-1}RQS^{-1} &0\\0& T_2}$, which are invertible. In other words, $U$ and $V$ are equivalent, hence EAOE, and thus SC.
\end{proof}

There are other cases with $\cX\simeq \cY$ where EAE and SC coincide, but we intend to return to this topic in a separate paper.\medskip

\paragraph{\bf Acknowledgments}
This work is based on research supported in part by the National Research Foundation of South Africa (NRF) and the DST-NRF Centre of Excellence in Mathematical and Statistical Sciences (CoE-MaSS). Any opinion, finding and conclusion or recommendation expressed in this material is that of the authors and the NRF and CoE-MaSS do not accept any liability in this regard.


\end{document}